\documentclass[10pt,oneside,a4paper,reqno]{amsart}  

\usepackage{amssymb,graphicx,mathrsfs,xcolor}
\usepackage[colorlinks=true, linkcolor=magenta, citecolor=cyan, urlcolor=blue]{hyperref}

\numberwithin{equation}{section}\newtheorem{theorem}{Theorem}[section]

\newtheorem{proposition}[theorem]{Proposition}\theoremstyle{remark}
\newtheorem{remark}{Remark}[section]
\theoremstyle{definition}
\newtheorem{definition}[theorem]{Definition}

\title[$n$-dimensional Dirac equation]
{Virial identity and dispersive estimates\\
for the $n$-dimensional Dirac equation}
\date{\today}    

\author{Federico Cacciafesta}
\address{Federico Cacciafesta: 
SAPIENZA --- Universit\`a di Roma,
Dipartimento di Matematica, 
Piazzale A.~Moro 2, I-00185 Roma, Italy}
\email{cacciafe@mat.uniroma1.it}


\begin{document}\maketitle
\begin{abstract}
  We extend to general dimension $n\ge2$ the virial identity
  proved in \cite{BoussaidDAnconaFanelli11-a} for the 3D
  magnetic Dirac equation. As an application we deduce smoothing and Strichartz
  estimates for an $n$-dimensional Dirac equation perturbed with
  a magnetic potential. To prove the smoothing estimates we use the multiplicator 
  technique already developed in \cite{BoussaidDAnconaFanelli11-a} and in \cite{FanelliVega09-a}.
\end{abstract}

\section{Introduction}\label{sec:introduction}  


The \emph{Dirac equation on $\mathbb{R}^{1+n}$} is a
constant coefficient, hyperbolic system of the form
\begin{equation}\label{eq1}
iu_t+\mathcal{D}u+m\beta u=0\\
\end{equation}
where $u:\mathbb{R}_t\times\mathbb{R}_x^n\rightarrow\mathbb{C}^M$,
the \emph{Dirac operator} is defined by
\begin{equation*}
  \mathcal{D}=
  i^{-1}\displaystyle\sum_{k=1}^n\alpha_k
  \frac{\partial}{\partial x_{k}}=
  i^{-1}(\alpha\cdot\nabla),
\end{equation*}
and the \emph{Dirac matrices}
$\alpha_{0}\equiv\beta$, $\alpha_{1},\dots,\alpha_{n}$
are a set of $M \times M$ hermitian matrices satisfying
the anti-commutation relations
\begin{equation}\label{matrices}
  \alpha_j\alpha_k+\alpha_k\alpha_j=2\delta_{jk}\mathbb{I}_M,\qquad
  0\le j,k\le n.
\end{equation}
The quantity $m\ge0$ is called the \emph{mass} and in
the classical 3D model is linked with the mass of a spin 1/2 particle.

\begin{remark}\label{rem:ndimdirac}
  For each dimension $n\ge1$ there exist different choices of $M$ and of
  matrices $\alpha_{j}$
  satisfying all of the above conditions; the original Dirac 
  equation corresponds to $n=3$, $M=4$, in which case the 4 matrices can
  be chosen from a well known set of 16 anticommuting matrices
  (see \cite{Thaller92-a}). A possible way to construct a familiy of matrices satisfying such properties is
  the following. \\
  For $n=1$ let
  $$
  \alpha_0^1=\left(\begin{array}{cc}0 & 1 \\1 & 0\end{array}\right),\qquad
  \alpha_1^1=\left(\begin{array}{cc}1 & 0 \\0 & -1\end{array}\right).
  $$
  For $n\geq2$ let
  $$
  \alpha_j^{(n)}=\left(\begin{array}{cc}0 & \alpha_j^{(n-1)} \\\alpha_j^{(n-1)} & 0\end{array}\right),\quad 
  j=0,...,n-1,\qquad
  \alpha^{(n)}_n=\left(\begin{array}{cc}I_n & 0 \\0 & -I_n\end{array}\right).
  $$
  Notice that in this case $M=2^n$ (for a more
  detailed analysis of general Dirac matrices, see
  \cite{MachiharaNakamuraOzawa04-a},
  \cite{KalfYamada01-a},
  \cite{OzawaYamauchi04-a})
\end{remark}

An easy consequence of the anticommutation relations is the identity
\begin{equation}\label{sq2}
  (i\partial_t-\mathcal{D}-m\beta)(i\partial_t+\mathcal{D}+m\beta)=
  (\Delta-m^2-\partial_{tt}^2)\mathbb{I}_M.
\end{equation}
which reduces the study of \eqref{eq1} to a corresponding
study of the Klein-Gordon equation, or the wave equation
in the massless case $m=0$. The analysis of the important
Maxwell-Dirac and Dirac-Klein-Gordon systems of 
quantum electrodynamics in
\cite{Bournaveas96-a}-
\cite{Bournaveas99-a}
was based on this method;
notice however that in the reduction
step some essential details of the structure may be lost,
as recently pointed out in
\cite{DanconaFoschiSelberg07-a},
\cite{DanconaFoschiSelberg07-b},
\cite{DanconaFoschiSelberg10-a}.

From \eqref{sq2} one can deduce in a straightforward way the
dispersive properties of the Dirac flow from the corresponding
properties of the wave-Klein-Gordon flow.
Based on this approach, an extensive theory of local and
global well posedness for nonlinear perturbations of
\eqref{eq1} was developed in
\cite{DiasFigueira87-a},
\cite{EscobedoVega97-a},
\cite{MachiharaNakamuraOzawa04-a},
\cite{MachiharaNakamuraNakanishi05-a};
see also 
\cite{DanconaFanelli07-a},
\cite{DanconaFanelli08-a} for a study of
the dispersive properties 
of the Dirac equation perturbed by a magnetic field.

The goal of this paper is to study the dispersive
properties of the system \eqref{eq1} perturbed by
a magnetic field, thus
extending to the $n$-dimensional setting the smoothing
and Strichartz estimates proved in 
\cite{BoussaidDAnconaFanelli11-a}
for the 3D magnetic Dirac equation. Denoting with
\begin{equation*}
  A(x)=(A^1(x),...,A^n(x)):\mathbb{R}^n\rightarrow\mathbb{R}^n
\end{equation*}
a static magnetic potential, the standard way to express its interaction
with a particle is by replacing the derivatives $\partial_{k}$
with their covariant counterpart $\partial_{k}-iA^{k}$,
thus obtaining the \emph{magnetic Dirac operator}
\begin{equation}\label{D}
  \mathcal{D}_A=
  i^{-1}\displaystyle\sum_{k=1}^n\alpha_k(\partial_k-iA^k)=
  i^{-1}\alpha \cdot \nabla_{A},\qquad
  \nabla_{A}=\nabla-iA(x).
\end{equation}
Here and in the following we denote with a dot
the scalar product of two vectors of operators:
\begin{equation*}
  (P_{1},\dots,P_{m})\cdot (Q_{1},\dots,Q_{m})=
  \sum_{j=1}^{m}P_{j}Q_{j}.
\end{equation*}
We shall also use the unified notation
\begin{equation}\label{H}
  \displaystyle\mathcal{H}=
  i^{-1}\alpha\cdot\nabla_A+m\beta=\mathcal{D}_A+m\beta
\end{equation}
to include both the massive and the massless case.

Thus we plan to investigate the dispersive properties of the flow $e^{it\mathcal{H}}f$ defined as the solution to the Cauchy problem
\begin{equation}\label{Hprob}
  iu_t(t,x)+\mathcal{H}u(t,x)=0,
  \qquad
  u(0,x)=f(x).
\end{equation}
It is natural to require that the operator $\mathcal{H}$ be selfadjoint. 
Several sufficient conditions are known for selfadjointness
(see \cite{Thaller92-a}). For greatest generality, we prefer to
make an abstract selfadjointness assumtpion; we also includes a
density condition which allows to approximate rough solutions with smoother 
ones, locally uniformly in time, and is easily verified in concrete cases.
The condition is the following:

\begin{quotation}
  SELF-ADJOINTNESS ASSUMPTION (A). The operator $\mathcal{H}$ is 
  essentially selfadjoint on $C^\infty_c(\mathbb{R}^n)$, and in 
  addition for initial data $f\in C^\infty_c(\mathbb{R}^n)$ the 
  flow $e^{it\mathcal{H}}f$ belongs at least
  to $C(\mathbb{R}, H^{3/2})$.
\end{quotation}

The main tool used here is the method of Morawetz multipliers, in the
version of \cite{DanconaFanelliVega10-a}, \cite{BoussaidDAnconaFanelli11-a}.
This method allows to partially overcome the smallness assumption on the
potential which was necessary for the perturbative approach of
\cite{DanconaFanelli08-a}. An additional advantage is that the
assumptions on the potential are expressed in terms of the
\emph{magnetic field} $B$ rather than the vector potential $A$;
indeed, $B$ is a physically measurable quantity while $A$
should be thought of as a mathematical abstraction.
We recall that in dimension 3 the magnetic field $B$ is defined as
\begin{equation*}
  B=\rm{curl} A.
\end{equation*}
In arbitrary dimension $n$, a natural generalization of the previous
definition is the following

\begin{definition}\label{B}
  Given a magnetic potential $A:\mathbb{R}^{n}\to \mathbb{R}^{n}$,
  the \emph{magnetic field} 
  $B:\mathbb{R}^n\rightarrow \mathcal{M}_{n\times n}(\mathbb{R})$ 
  is the matrix valued function
  \begin{equation*}
    B=DA-DA^t,\qquad 
    B^{jk}=\displaystyle\frac{\partial A^j}
    {\partial x^k}-\frac{\partial A^k}{\partial x^j}
  \end{equation*}
  and its \emph{tangential component}
  $B_\tau=\mathbb{R}^n\rightarrow\mathbb{R}^n$ is defined as
  \begin{equation*}
    B_\tau=\displaystyle\frac{x}{|x|}B.
  \end{equation*}
  Notice indeed that $B_{\tau}(x)$ is orthogonal to $x$
  for all $x$.
\end{definition}

\begin{remark}
  The previous definition reduces to the standard one in
  dimension $n=3$: indeed the matrix $B$ satisfies
  for all $v\in \mathbb{R}^{3}$
  \begin{equation*}
    Bv=\rm{curl}A\wedge v
  \end{equation*}
  and in this sense $B$ can be identified with $\mathrm{curl}A$. 
  Notice also that
  \begin{equation*}
    B_\tau=\displaystyle\frac{x}{|x|}\wedge\mathrm{curl} A.
  \end{equation*}
\end{remark}

Our first result is the following (formal) virial identity for the 
$n$-dimensional magnetic Dirac equation \eqref{Hprob}:

\begin{theorem}[Virial identity]\label{virialdir}
  Assume that the operator $\mathcal{H}$ defined
  in \eqref{H} satisfies (A), and let 
  $\phi:\mathbb{R}^n\rightarrow \mathbb{R}$ be a real valued function.
  Then any solution $u(t,x)$ of \eqref{Hprob}
  satisfies the formal virial identity 
  \begin{equation}\label{vir}  
  \begin{split}
     & 
     2\int_{\mathbb{R}^n}
       \nabla_Au \cdot D^2\phi \cdot\overline{\nabla_Au}-\frac{1}{2}
     \int_{\mathbb{R}^n}|u|^2\Delta^2\phi
     +2\int_{\mathbb{R}^n}
     \Im\left(
        u \nabla\phi \cdot B\cdot\overline{\nabla_Au}
     \right)
     +
  \\
     & 
     +\int_{\mathbb{R}^n}
       \overline{u}\cdot
       \sum_{j<k}\alpha_j\alpha_k(\nabla\phi
          \cdot\nabla B^{jk})u
          =
     -\frac{d}{dt}
     \int_{\mathbb{R}^n}
     \Re
     \left(u_t(2\nabla\phi\cdot\overline{\nabla_Au}
           +\overline{u}\Delta\phi)\right).
  \end{split}
  \end{equation}
\end{theorem}

\begin{remark}\label{radial}
  If $\phi=\phi(|x|)$ is a radial function, as we shall always assume 
  in the following, the virial identity can be considerably simplified.
  In particular, notice that
  $$
  \sum_{j<k}\alpha_j\alpha_k(\nabla\phi\cdot
  \nabla B^{jk})=
  \phi'(|x|)\sum_{j<k}\alpha_j\alpha_k\partial_rB^{jk}.
  $$
\end{remark}

As a direct consequence of the previous virial identity, 
we can prove a smoothing estimate for the $n$-dimensional 
magnetic Dirac equation \eqref{Hprob}.
\\In the following we shall denote respectively with 
$\nabla_A^r u$ and $\nabla^\tau_Au$ 
the radial and tangential components of the covariant gradient, 
namely
$$
\nabla_A^ru:=\displaystyle\frac{x}{|x|}\cdot\nabla_Au,\qquad
\nabla_A^\tau u:=\nabla_Au-\frac{x}{|x|}\cdot\nabla_A^ru
$$
so that
$$
|\nabla^r_A u|^2+|\nabla^\tau_A u|^2=|\nabla_A u|^2.
$$
We shall use the notation
\begin{equation*}
  [B]_{1}=\sum_{j,k=1}^{n}|B^{jk}|
\end{equation*}
to denote the $\ell^{1}$ norm of a matrix (i.e.~the sum of the
absolute values of its entries), and we shall measure the size
of matrix valued functions using norms like
\begin{equation*}
  \|B\|_{L^{\infty}}=\|[B(x)]_{1}\|_{L^{\infty}_{x}}
\end{equation*}

Then we have:

\begin{theorem}[Smoothing estimates]\label{smootheo}
  Let $n\geq4$.
  Let the operator $\mathcal{H}$ defined in (\ref{H}) satisfies
  assumption $(A)$. Let $B=DA-DA^{t}=B_{1}+B_{2}$
  with $B_2\in L^\infty$, and assume that 
  \begin{equation}\label{hp1}
   \displaystyle
    |B_\tau(x)|\leq\frac{C_1}{|x|^2},\quad 
    \frac{1}{2}[\partial_rB(x)]_1\leq\frac{C_2}{|x|^3}
 \end{equation}
 for all $x\in\mathbb{R}^n$ and 
for some constants $C_1$, $C_2$ such that
  \begin{equation}\label{hpC}
   \displaystyle C^2_1+2C_2\leq\frac{2}{3}(n-1)(n-3)
  \end{equation}
  Assume moreover that
  $$
C_0=\||x|^2B_1\|_{L^\infty(\mathbb{R}^n)}<\frac{(n-2)^2}{4}.
  $$
  Finally, in the massless case restrict the choice to
  $B_{1}=B$, $B_{2}=0$ in the above assumptions.
  
  Then for all $f\in L^2$ the following smoothing estimate holds
  \begin{equation}\label{smooth}
    \sup_{R>0}\int_{-\infty}^{+\infty}\int_{|x|
    \leq R}|e^{it\mathcal{H}}f|^2dxdt\lesssim\|f\|_{L^2}^2.
  \end{equation}
\end{theorem}

The limitation to $n\ge3$ space dimensions is intrinsic in the
multiplier method; low dimensions $n=1,2$ require 
a different approach (see e.g.~\cite{DanconaFanelli06-a} for
a general result in dimension 1). In the present paper 
we shall only deal with the case $n\geq4$, the 3-dimensional case being exaustively 
discussed in \cite{BoussaidDAnconaFanelli11-a}. Notice that, as it often occurs, the 
three dimensional case yields different hypothesis on the potential, being slightly different 
the multiplicator that one needs to consider.

A natural application of the smoothing estimate \eqref{smooth}
is to derive Strichartz estimates for the perturbed flow
$e^{it\mathcal{H}}f$, both in the massless and massive case.
Our concluding result is the following:

\begin{theorem}[Strichartz estimates]\label{strichartz}
  Let $n\ge4$.
  Assume $\mathcal{H}$, $A$, $B$ are as in 
  Theorem \ref{smootheo}, and in addition assume that
  \begin{equation}\label{hpA}
    \sum_{j\in\mathbb{Z}}2^j\sup_{|x|\cong 2^j}|A|<\infty.
  \end{equation}
  Then the
  perturbed Dirac flow satisfies the Strichartz estimates
  \begin{equation}\label{pstrich1}
    \||D|^{\frac{1}{q}-\frac{1}{p}-\frac{1}{2}}
      e^{it\mathcal{H}}f\|
      _{L^pL^{q}}\lesssim\|f\|_{L^2}
  \end{equation}
  where, in the massless case $m=0$, the couple $(p,q)$
  is any wave admissibile, non-endpoint couple i.e.~such that
  \begin{equation}\label{adm1}
    \frac{2}{p}+\frac{n-1}{q}=\frac{n-1}{2},\qquad2
    < p\leq\infty\qquad\frac{2(n-1)}{n-3}> q\geq 2,
  \end{equation}
  while in the massive case the same bound holds
  for all Schr\"odinger adimmissible couple, non-endpoint
  $(p,q)$, i.e. such that
  \begin{equation}\label{adm2}
    \frac{2}{p}+\frac{n}{q}=\frac{n}{2},\qquad2<
    p\leq\infty\qquad\frac{2n}{n-2}> q\geq 2.
  \end{equation}
\end{theorem}

The paper is organized as follows: in Section 2 we shall prove Theorem
\ref{virialdir}, deriving it from a classical virial identity for the 
wave equation (see Theorem \ref{virialwave}) plus the algebric structure 
of the Dirac operator.
In Section 3 we shall use the multiplicator technique to prove 
the smoothing estimate (\ref{smooth}) from Theorem 
\ref{virialdir}.
Finally in Section 4 we shall derive the Strichartz estimates of 
Theorem \ref{strichartz} by a perturbative argument based on
the smoothing estimates. 
Section 5 is devoted to the proof of a magnetic Hardy inequality 
for the Dirac operator, needed at several steps in the proof
of the previous theorems.

\section{Proof of the virial identity}

\noindent Let $u$ be a solution to equation (\ref{eq1}). Using identity 
$$
0=(i\partial_t-\mathcal{H})(i\partial_t+\mathcal{H})u=(-\partial_{tt}-\mathcal{H}^2)u,
$$
we see that $u$ solves the Cauchy problem for a magnetic wave equation:
\begin{equation}\label{pb1}
\begin{cases}
u_{tt}+\mathcal{H}^2u=0\\
u(0)=f\\
u_t(0)=i\mathcal{H}f.
\end{cases}
\end{equation}
In \cite{BoussaidDAnconaFanelli11-a} the following general result
was proved for a solution $u(t,x)$ of wave-type equations:

\begin{theorem}[\cite{BoussaidDAnconaFanelli11-a}]
  \label{virialwave}
  Let $L$ be a selfadjoint operator on $L^{2}(\mathbb{R}^{n})$,
  and let $u(t,x)$ be a solution of the equation
  $$
  u_{tt}(t,x)+Lu(t,x)=0.
  $$  
  Let $\phi:\mathbb{R}^n\rightarrow\mathbb{R}$ and define the quantity
  \begin{equation}\label{theta}
  \Theta(t)=(\phi u_t,u_t)+\mathcal{R}((2\phi L-L\phi)u,u).
  \end{equation}
  Then $u(t,x)$ satisfies the formal virial identities
  \begin{equation}\label{vir1}
    \dot{\Theta}(t)=\mathcal{R}([L,\phi]u,u_t)
    \end{equation}
    \begin{equation}\label{vir2}
    \ddot{\Theta}(t)=\displaystyle-\frac{1}{2}([L,[L,\phi]]u,u).
  \end{equation}
\end{theorem}

In order to apply this proposition to our case we thus need to compute explicitly the commutators in (\ref{vir1}), (\ref{vir2}) with the choice $L=\mathcal{H}^2$.
We begin by expanding the square
\begin{equation*}%
  \mathcal{H}^2=
  (\mathcal{H}_{0}-\alpha \cdot A)^{2}=
  \mathcal{H}^2_0-\mathcal{H}_0(\alpha\cdot A)-(\alpha\cdot A)
  \mathcal{H}_0+(\alpha\cdot A)(\alpha\cdot A),
\end{equation*}
and we recall that the unperturbed part of the operator
\begin{equation*}
  \mathcal{H}_{0}=
  \mathcal{D}+m \beta=
  i^{-1}\alpha \cdot \nabla+m \beta
\end{equation*}
satisfies
$$
\mathcal{H}^2_0=(m^2-\Delta)\mathbb{I}_M.
$$
Since $\beta$ anticommutes with each $\alpha_{j}$ we get
\begin{equation}\label{Hsquare}
  \mathcal{H}^2=
  \mathcal{H}_{0}^{2}-
  i^{-1}(\alpha \cdot \nabla)(\alpha \cdot A)-
  i^{-1}(\alpha \cdot A)(\alpha \cdot \nabla)+
  (\alpha\cdot A)(\alpha\cdot A).
\end{equation}

We need a notation to distinguish the composition of the
operators (multiplication by) $A_{k}$ and $\partial_{j}$,
which we shall denote with $\partial_{j}\circ A^{k}$, i.e.,
\begin{equation*}
  \partial_{j}\circ A^{k}u= \partial_{j}(A^{k}u)
\end{equation*}
and the simple derivative $\partial_{j}A^{k}$.
After a few steps we obtain (we omit for simplicity
the factor $\mathbb{I}_{M}$ in diagonal operators)
\begin{equation*}
  \mathcal{H}^2=\mathcal{H}_0^2+
     i(\nabla\cdot A)+
     i(A\cdot\nabla)+
     |A|^2+
     i\sum_{j\neq k}^n\alpha_j\alpha_k
     (\partial_j\circ A^k+A^j\partial_k).
\end{equation*}
or equivalently
\begin{equation}\label{Hsquare2}
  \mathcal{H}^2=
     (m^{2}-\Delta_{A})+
     i\sum_{j\neq k}^n\alpha_j\alpha_k
     (\partial_j\circ A^k+A^j\partial_k),
\end{equation}
where
\begin{equation*}
  \Delta_{A}=(\nabla-i A)^{2}=\nabla_{A}^{2}.
\end{equation*}
Now we observe that
\begin{equation*}
\begin{split}
  \sum_{j\neq k}
    \alpha_j\alpha_k
    &
    (\partial_j\circ A^k+A^j\partial_k)
  \\
  =&\sum_{j< k}
  \alpha_j\alpha_k
  [(\partial_j\circ A^k+A^j\partial_k)-
        (\partial_k\circ A^j+A^k\partial_j)]=
    \\
  =&\sum_{j<k}
    \alpha_j\alpha_k
    (\partial_jA^k-\partial_kA^j)
    \\
  =&\sum_{j<k}
    \alpha_j\alpha_kB^{jk}=
    \\ 
  =&\frac14\sum_{j,k=1}^{n}(\alpha_{j}\alpha_{k}-\alpha_{k}\alpha_{j})
     B^{jk}
\end{split}
\end{equation*}
since $B$ is skewsymmetric.
If we introduce the matrix $S=[S_{jk}]$ whose entries are the matrices
\begin{equation*}
  S_{jk}=\frac14(
  \alpha_{j}\alpha_{k}-\alpha_{k}\alpha_{j})
  \equiv \frac12\alpha_{j}\alpha_{k}
\end{equation*}
and we use the notation 
\begin{equation*}
  [a_{jk}]\cdot[b_{jk}]=\sum_{j,k=1}^{n}a_{jk}b_{jk}
\end{equation*}
for the scalar product of matrices, the above identity can be
compactly written in the form
\begin{equation*}
  \sum_{j\neq k}
    \alpha_j\alpha_k
    (\partial_j\circ A^k+A^j\partial_k)=
  S \cdot B.
\end{equation*}
In conclusion we have proved that
\begin{equation}\label{Hsquare3}
  \mathcal{H}^2=(m^2-\Delta_A)\mathbb{I}_M
  +iS \cdot B
\end{equation}
and hence for the massless case
\begin{equation}\label{Hsquare3bis}
  \mathcal{D}^2_A=-\Delta_A\mathbb{I}_M 
  +iS \cdot B.
\end{equation}
Thus the commutator with $\phi$ reduces to
$$
[\mathcal{H}^2,\phi]=[m^2,\phi]-[\Delta_A,\phi]+i[S\cdot B,\phi]=-[\Delta_A,\phi].
$$
Using the Leibnitz rule
$$
\nabla_A(fg)=g\nabla_Af+f\nabla g,
$$
we arrive at the explicit formula
\begin{equation}\label{comm1}
[\mathcal{H}^2,\phi]=-[\Delta_A,\phi]=-2\nabla \phi\cdot\nabla_A-(\Delta\phi).
\end{equation}
Recalling (\ref{theta}) and (\ref{vir1}) we thus obtain
\begin{equation}\label{comm1bis}
\dot{\Theta}(t)=-\Re\displaystyle\int_{\mathbb{R}^n}u_t(2\nabla\phi\cdot\overline{\nabla_Au}+\overline{u}\Delta\phi).
\end{equation}
We now turn to the second commutator. By formulas (\ref{Hsquare3}) and (\ref{comm1}) we have
\begin{equation}\label{comm2}
[\mathcal{H}^2,[\mathcal{H}^2,\phi]]=[\Delta_A,[\Delta_A,\phi]]
-i[S\cdot B,[\Delta_A,\phi]].
\end{equation}
The first commutator is well known and was computed e.g.~in
\cite{DanconaFanelliVega10-a}; taking formula (2.19) there
(with $V \equiv0$) we obtain
\begin{equation}\label{2.19}
  (u,[\Delta_A,[\Delta_A,\phi]])=
  4\int_{\mathbb{R}^n}
  \nabla_AuD^2\phi\overline{\nabla_Au}-
  \int_{\mathbb{R}^n}|u|^2\Delta^2\phi+
\end{equation}
$$
  +4\Im\int_{\mathbb{R}^n}
  u\nabla\phi B_\tau\cdot\overline{\nabla_Au}.
$$
By (\ref{comm1}) the last term in (\ref{comm2}) becomes
$$
[S\cdot B,[\Delta_A,\phi]]=2[S\cdot B,\nabla\phi\cdot\nabla_A]=
$$
$$
=2(S\cdot B\nabla\phi\cdot\nabla_A-\nabla\phi\cdot\nabla_AS\cdot B)=
$$
$$
=\sum_{j<k}\alpha_j\alpha_k B^{jk}\nabla\phi\cdot\nabla_A-\nabla\phi\cdot\nabla_A\sum_{j<k}\alpha_j\alpha_k B^{jk}=
$$
$$
=\displaystyle\sum_{j<k}\alpha_j\alpha_k[B^{jk},\nabla\phi\cdot\nabla_A]=
$$
\begin{equation}\label{comm22}
=\displaystyle-\sum_{j<k}\alpha_j\alpha_k(\nabla\phi\cdot\nabla B^{jk}).
\end{equation}
Identity (\ref{vir}) then follows from (\ref{vir2}), (\ref{comm1bis}), (\ref{comm2}), (\ref{2.19}) and (\ref{comm22}).

\section{Smoothing estimates}

We shall use the following radial multiplier (for a detailed description see \cite{FanelliVega09-a}, \cite{BoussaidDAnconaFanelli11-a}):
\begin{equation}\label{tilde}
\tilde{\phi}_R(x)=\phi(x)+\varphi_R(x)
\end{equation}
where 
$$
\phi(x)=|x|
$$
for which we have
$$
\phi'(r)=1,\qquad\phi''(r)=0,\qquad
\Delta^2\phi(r)=\displaystyle-\frac{(n-1)(n-3)}{r^3}
$$
with the notation $r=|x|$, and $\varphi_R$ is the rescaled $
\displaystyle\varphi_R(r)=R\varphi_0(\frac{r}{R}),
$
of the multiplier 
\begin{equation}\label{varphi1}
\varphi_0(r)=\int_0^r\varphi'(s)ds
\end{equation}
where
\begin{equation}\label{varphi2}
\varphi_0'(r)=
\begin{cases}
\frac{n-1}{2n}r,\quad r\leq1
\\
\frac{1}{2}-\frac{1}{2nr^{n-1}},\quad r>1
\end{cases}
\end{equation}
and so
$$
\varphi_0''(r)=
\begin{cases}
\frac{n-1}{2n},\quad r\leq1
\\
\frac{n-1}{2nr^n},\quad r>1.
\end{cases}
$$
Thus we have 
\begin{equation}\label{der1}
\varphi_R'(r)=
\begin{cases}
\frac{(n-1)r}{2nR},\quad r\leq R
\\
\frac{1}{2}-\frac{R^{n-1}}{2nr^{n-1}},\quad r>R
\end{cases}
\end{equation}
\begin{equation}\label{der2}
\varphi_R''(r)=
\begin{cases}
\frac{1}{R}\frac{n-1}{2n},\quad r\leq R
\\
\frac{1}{R}\frac{R^n(n-1)}{2nr^n},\quad r>R
\end{cases}.
\end{equation}
\begin{equation}\label{bil}
\Delta^2\varphi_R=\displaystyle-\frac{n-1}{2R^2}\delta_{|x|=R}-\frac{(n-1)(n-3)}{2r^3}\chi_{[R,+\infty)}.
\end{equation}
Notice that $\varphi_R'$, $\varphi_R''$, $\Delta\varphi_R\geq0$ and moreover
$
\displaystyle\sup_{r\geq 0}\varphi'(r)\leq\frac{1}{2}.
$
\\
Thus it's easy to show the bounds for the derivatives of the perturbed multiplier
\begin{equation}\label{boundsmult}
\displaystyle\sup_{r\geq 0}\tilde{\phi}_R'\leq\frac{3}{2},\qquad\Delta\tilde{\phi}_R\leq\frac{n}{r}.
\end{equation}
\\We separate the estimates of the LHS and the RHS of (\ref{vir})\\\\
\textbf{Estimate of the RHS of (\ref{vir})}\\
Consider the expression
$$
\displaystyle\int_{\mathbb{R}^n}u_t(2\nabla\phi\cdot\overline{\nabla_A u}+u\Delta\phi)=
(u_t,2\nabla\phi\cdot\nabla_Au+\overline{u}\Delta\phi)_{L^2}
$$
appearing at the right hand side of (\ref{vir}). Since $u$ solves the equation we can replace $u_t$ with
$$
u_t=-i\mathcal{H}u=-im\beta u-i\mathcal{D}_Au.
$$
By the selfadjointess of $\beta$ it is easy to check that
$$
\Re[-im(\beta u,2\nabla\phi\cdot\nabla_Au)-im(\beta u,\Delta\phi u)]=0
$$
so that
$$
\Re[(u_t,2\nabla\phi\cdot\nabla_Au+u\Delta\phi)=2\mathcal{I}(\mathcal{D}_A u,\nabla\phi\cdot\nabla_A u)]+\mathcal{I}(\mathcal{D}_Au,\Delta\phi u)
$$
and by Young inequality we obtain
\begin{equation}\label{3.7}
\left|\Re\left(\int_{\mathbb{R}^n}u_t(2\nabla\phi\cdot\overline{\nabla_A u}+u\Delta\phi)\right)\right|\leq
\frac{3}{2}\|\mathcal{D}_Au\|_{L^2}^2+\|\nabla\phi\cdot\nabla_Au\|^2_{L^2}+\frac{1}{2}\|u\Delta\phi\|^2_{L^2}.
\end{equation}
Now we put in (\ref{3.7}) the multiplicator $\tilde{\phi}$ defined in (\ref{tilde}). From the boundedness of $\varphi$ and the magnetic Hardy inequality (\ref{hardyineq})
 we have, with the choice $\varepsilon=(n-2)^2-4C_0$ which is positive in virtue of the assumption $C_0<(n-2)^2/4$, 
\begin{equation}\label{3.8}
\|\nabla\tilde{\phi}\cdot\nabla_Au\|^2_{L^2}\leq \frac{3}{2}\displaystyle\frac{1}{(n-2)^2-4C_0}\|\mathcal{D}_Au\|^2_{L^2}.
\end{equation}
The third term in (\ref{3.7}) can be estimated again using Hardy inequality with
\begin{equation}\label{3.9}
\|u\Delta\tilde{\phi}\|^2_{L^2}\leq\displaystyle\frac{4n}{(n-2)^2-4C_0}\|\mathcal{D}_Au\|^2_{L^2}.
\end{equation}
Summing up, by (\ref{3.7}), (\ref{3.8}) and (\ref{3.9}) we can conclude
\begin{equation}\label{3.10}
\left| \Re\left(\int_{\mathbb{R}^n}u_t(2\nabla\phi\cdot\overline{\nabla_Au}+\overline{u}\Delta\phi)\right)\right|\leq c(n)\|\mathcal{D}_Au\|^2_{L^2}.
\end{equation}\\\\

\textbf{Estimate of the LHS of (\ref{vir})}\\
We shall make use of the following identity, that holds in every dimension:
\begin{equation}\label{dec}
\nabla_AuD^2\phi\overline{\nabla_Au}=\frac{\phi'(r)}{r}|\nabla^\tau_Au|^2+\phi''(r)|\nabla_A^ru|^2.
\end{equation}
For the seek of simplicity, we divide this part in two steps, first considering just the multiplier $\phi(r)=r$, for which the calculations turn out fairly straightforward, and then perturbating it to $\tilde{\phi}$.
\\\\\emph{Step 1}\\
With the choice $\phi(r)=r$, by (\ref{dec}) we can rewrite the LHS of (\ref{vir}) as follows:
\begin{equation}\label{vir2}
\displaystyle2\int_{\mathbb{R}^n}\frac{|\nabla^\tau_A u|^2}{|x|}dx+\frac{(n-1)(n-3)}{2}\int_{\mathbb{R}^n}\frac{|u|^2}{|x|^3}dx+
\end{equation}
$$
\displaystyle+2\int_{\mathbb{R}^n}\Im(uB_\tau\cdot\overline{\nabla_Au})dx
+\int_{\mathbb{R}^n}
       \overline{u}\cdot
       \sum_{j<k}\alpha_j\alpha_k
          \partial_r B^{jk}u.
$$
The first thing to be done is to prove this quantity to be positive. For what concerns the perturbative term, assuming that $$
|B_\tau|\leq\displaystyle\frac{C_1}{|x|^2}
$$
we have
\begin{equation}\label{Bt}
\displaystyle-\left|2\int_{\mathbb{R}^n}\Im(uB_\tau\cdot\overline{\nabla_Au})dx\right|\geq
-2\left(\int_{\mathbb{R}^n}\frac{|u|^2}{|x|^3}dx\right)^\frac{1}{2}
\left(\int_{\mathbb{R}^n}|x|^3|B_\tau|^2|\nabla^\tau_A u|^2dx\right)^\frac{1}{2}
\end{equation}
$$
\geq -2C_1K_1K_2,
$$
where
$$
K_1=\displaystyle\left(\int_{\mathbb{R}^n}\frac{|u|^2}{|x|^3}dx\right)^\frac{1}{2}
$$
$$
K_2=\displaystyle\left(\int_{\mathbb{R}^n}\frac{|\nabla^\tau_A u|^2}{|x|}dx\right)^\frac{1}{2}.
$$
Analogously, assuming
$$
 \Bigl\|\sum_{j<k} \alpha_{j}\alpha_{k}
    \partial_{r}B^{jk}(x)\Bigr\|_{M \times M}\le 
    \frac12
    [\partial_{r}B(x)]_{1}\leq \frac{C_2}{|x|^3}
$$
(recall that here $\|\cdot\|_{M\times M}$ denotes the operator norm of $M\times M$ matrices and $[\cdot]_{1}$ denotes 
the sum of absolute values of the entries of a matrix) 
we have
\begin{equation}\label{Br}
- \left|\int_{\mathbb{R}^n}
    \overline{u}\cdot
    \sum_{j<k}\alpha_j\alpha_k
      \partial_rB^{jk}udx
  \right|\geq
  -\int|u|^{2}
  \Bigl\|\sum_{j<k} \alpha_{j}\alpha_{k}
    \partial_{r}B^{jk}\Bigr\|_{M \times M}dx\geq -C_2K_1^2
\end{equation}
where $K_1$ is as before. Thus we have reached the following estimate
$$
\displaystyle2\int_{\mathbb{R}^n}\frac{|\nabla^\tau_A u|^2}{|x|}dx+\frac{(n-1)(n-3)}{2}\int_{\mathbb{R}^n}\frac{|u|^2}{|x|^3}dx+
$$
$$
\displaystyle+2\int_{\mathbb{R}^n}\Im(uB_\tau\cdot\overline{\nabla_Au})dx
+\int_{\mathbb{R}^n}
       \overline{u}\cdot
       \sum_{j<k}\alpha_j\alpha_k
          \partial_r B^{jk}u\geq
$$
$$
\geq 2K_2^2-2C_1K_1K_2-C_2K^2_1+\displaystyle\frac{(n-1)(n-3)}{2}K_1^2=:C(C_1,C_2,K_1,K_2).
$$
As usual, we want to optimize the condition on the constants $C_1$, $C_2$ under which the quantity $C$ is positive for all $K_1$, $K_2$. Fixing $K_1=1$ and requiring that
$$
\displaystyle\left(\frac{(n-1)(n-3)}{2}-C_2\right)K_1^2-2C_1K_1+2\geq0
$$
we can easily conclude that the resulting condition on the constants is given by
\begin{equation}\label{condC}
C_1^2+2C_2\leq(n-1)(n-3).
\end{equation}
Thus, if condition (\ref{condC}) is satisfied, we have that the quantity in (\ref{vir2}) is positive.
\\\\
\emph{Step 2}
\\We now perturb the multiplier to complete the proof. We thus put the multiplier $\tilde{\phi}_R$ as defined in (\ref{tilde}) in the LHS of (\ref{vir}), and repeat exactly the same calculations as in Step 1. With this choice, estimates (\ref{Bt}) and (\ref{Br}) still hold with the rescaled constants $\widetilde{C}_1=\frac{3}{2}C_1$, $\widetilde{C}_2=\frac{3}{2}C_2$ (since, as we have noticed, $\tilde\phi_R\leq\frac{3}{2}$), and thus the final conditions on the potential turn to be (\ref{hp1}), (\ref{hpC}).

Thus putting all together , taking the supremum over $R>0$, integrating in time and dropping the corresponding nonnegative terms we have reached the estimate
\begin{equation}\label{3.17}
\displaystyle2\int_{-T}^Tdt\int_{\mathbb{R}^n}\nabla_AuD^2\phi\overline{\nabla_Au}-\frac{1}{2}\int_{-T}^Tdt
\int_{\mathbb{R}^n}|u|^2\Delta^2\phi+
\end{equation}
$$
\displaystyle2\mathcal{I}\int_{-T}^Tdt\int_{\mathbb{R}^n}u\phi'B_\tau\cdot\overline{\nabla_Au}+\int_{-T}^Tdt\int_{\mathbb{R}^n}|u|^2\sum_{j<k}\alpha_j\alpha_k(\nabla\phi\cdot\nabla B^{jk})\geq
$$
$$
\geq\displaystyle
\sup_{R>0}\frac{1}{R}\int_{-T}^Tdt\int_{|x|\leq R}|\nabla_A u|^2dx\geq
$$
$$
\displaystyle\geq
\sup_{R>0}\frac{1}{R}\int_{-T}^Tdt\int_{|x|\leq R}|\mathcal{D}_A u|^2dx
$$
where in the last step we have used the pointwise inequality $|\mathcal{D}_Au|\leq|\nabla_A u|$. We now integrate in time the virial idetity on $[-T,T]$, and using (\ref{3.17}) and (\ref{3.10}) we obtain
\begin{equation}\label{3.18}
\displaystyle
\sup_{R>0}\frac{1}{R}\int_{-T}^Tdt\int_{|x|\leq R}|\mathcal{D}_A u|^2dx\lesssim\|\mathcal{D}_Au(T)\|^2_{L^2}+\|\mathcal{D}_Au(-T)\|^2_{L^2}.
\end{equation}
Let us now consider the range of $D_A$: from proposition (\ref{Hardy}) we have that for $C_0<(n-2)^2/4$ $0\not\in \ker(D_A)$, so $\rm{ran}(D_A)$ is either $L^2$ or it is dense in $L^2$. Fix now an arbitrary $g\in\rm{ran(D_A)}$, there exists $f\in D(D_A)=D(\mathcal{H})$ such that $D_Af=g$. We then consider the solution $u(t,x)$ to the problem
$$
\begin{cases}
iu_t=-m\beta u+\mathcal{D}_Au\\
u(0,x)=f(x)\\
\end{cases}
$$
with opposite mass, and notice that $u$ satisifies (\ref{3.18}) since no hypothesis on the sign of the mass $m$ have been used for it. If we thus apply to this equation the operator $\mathcal{D}_A$ we obtain, by the anticommutation rules,
$$
\begin{cases}
i(\mathcal{D}_Au)_t=\beta m(\mathcal{D}_Au)+\mathcal{D}_A(\mathcal{D}_Au)=0\\
\mathcal{D}_Au(0,x)=\mathcal{D}_Af(x)\\
\end{cases}
$$
or, in other words, the function $v=\mathcal{D}_Au$ solves the problem
$$
\begin{cases}
iv_t=\mathcal{H}v\\
v(0,x)=g\\
\end{cases}
$$
so that $v=e^{it\mathcal{H}}g$. Substituting in (\ref{3.18}) and letting $T\rightarrow\infty$ we conclude that
$$
\displaystyle\int_{-\infty}^\infty\|e^{it\mathcal{H}}f\|_{X}^2\lesssim\|g\|_{L^2}
$$
that is exactly (\ref{smooth}) for $g\in \rm{ran}(D_A)$, which is as we have noticed dense in $L^2$. Density arguments conclude the proof.

\section{Proof of the Strichartz estimates}

We begin by recalling the
Strichartz estimates for the free Dirac flow, both
in the massless and in the massive case. They are a direct consequence
of the corresponding estimates for the wave and Klein-Gordon
equations:

\begin{proposition}\label{free1}
Let $n\geq 3$. Then the following Strichartz estimates hold:

(i) in the massless case, for any wave admissible couple
$(p,q)$ (see \eqref{adm1})
\begin{equation}\label{fstrich1}
  \||D|^{\frac{1}{q}-\frac{1}{p}-\frac{1}{2}}
    e^{it\mathcal{D}}f\|
  _{L^pL^{q}}\lesssim \|f\|_{L^2};
\end{equation}

(ii) in the massive case, for any Schr\"odinger admissible couple
$(p,q)$ (see \eqref{adm2})
\begin{equation}\label{fstrich2}
  \||D|^{\frac{1}{q}-\frac{1}{p}-\frac{1}{2}}
    e^{it(\mathcal{D}+\beta)}f\|_{L^pL^{q}}
  \lesssim \|f\|_{L^2}.
\end{equation}
\end{proposition}

\begin{proof}
We restrict the proof to the case $n\geq$, refering to \cite{BoussaidDAnconaFanelli11-a} for an exaustive 
proof of the 3-dimensional case.\\
Recalling identity (\ref{sq2}) we immediately have that $u(t,x)=e^{it\mathcal{D}}f$ and $v(t,x)=e^{it\mathcal({D}+\beta)}$ 
satisfy the two Cauchy problems
\begin{equation}\label{syst1}
\begin{cases}
u_{tt}-\Delta u=0\\
u(0,x)=f(x)\\
u_t(0,x)=i\mathcal{D}f,
\end{cases}
\end{equation}
\begin{equation}\label{syst2}
\begin{cases}
v_{tt}-\Delta  v+mv=0\\
v(0,x)=f(x)\\
v_t(0,x)=i(\mathcal{D}+\beta)f,
\end{cases}
\end{equation}
and so each component of the $M$-dimensional vectors $u$ and $v$ 
satisfy the same Strichartz estimates as for the $n$-dimensional 
wave equation and Klein-Gordon equation respectively.
Thus case (i) follows from the standard estimates proved in
\cite{GinibreVelo95-a} and \cite{KeelTao98-a}, while
case (ii) follows from similar techniques (the details can be
found e.g.~in the Appendix of \cite{DanconaFanelli08-a}).
\end{proof}

We turn now to the perturbed flow. 
In the massless case, from the Duhamel formula we can write
\begin{equation}\label{duh}
  u(t,x)\equiv e^{it \mathcal{D}_{A}}f
  =e^{it\mathcal{D}}f+
  \int_0^te^{i(t-s)\mathcal{D}}\alpha\cdot Au(s)ds.
\end{equation}
The term $e^{it\mathcal{D}}f$ can be directly estimated with 
(\ref{fstrich1}). For the perturbative term we follow the 
Keel-Tao method \cite{KeelTao98-a}:
by a standard application of the Christ-Kiselev Lemma, 
since we only aim at 
the non-endpoint case, it is sufficient to estimate 
the untruncated integral
$$
\int e^{i(t-s)\mathcal{D}}\alpha\cdot Au(s)ds=
e^{it\mathcal{D}}\int e^{-is\mathcal{D}}\alpha\cdot Au(s)ds.
$$
Using again (\ref{fstrich1}) we have 
\begin{equation}\label{christ}
  \left\||D|^{\frac{1}{q}-\frac{1}{p}-\frac{1}{2}}
    e^{it\mathcal{D}}\int e^{-is\mathcal{D}}\alpha\cdot Au(s)ds 
  \right\|_{L^pL^{q}}
  \lesssim
  \left\|\int e^{-is\mathcal{D}}\alpha\cdot Au(s)ds \right\|_{L^2}.
\end{equation}
Now we use the dual form of the smoothing estimate (\ref{smooth}), i.e.
\begin{equation}\label{dual}
\displaystyle\left\|\int e^{-is\mathcal{D}}\alpha\cdot Au(s)ds \right\|_{L^2}\leq
\sum_{j\in\mathbb{Z}}2^\frac{j}{2}\||A|\cdot|u|\|_{L^2_tL^2(|x|\cong2^j)},
\end{equation}
where we have used the dual of the Morrey-Campanato norm 
as in \cite{PerthameVega99-a}.
Hence by H\"older inequality, hypothesis (\ref{hpA}) 
and estimate (\ref{smooth}) we have
\begin{equation}\label{fin}
\displaystyle\sum_{j\in\mathbb{Z}}2^\frac{j}{2}\||A|\cdot|u|\|_{L^2_tL^2(|x|\cong2^j)}\leq
\sum_{j\in\mathbb{Z}}2^j\sup_{|x|\cong 2^j}|A|\cdot\sup_{j\in\mathbb{Z}}\|u\|_{L^2_tL^2(|x|\cong2^j)}\lesssim\|f\|_{L^2}
\end{equation}
which proves (\ref{pstrich1}). 
The proof in the massive case is exactly the same.

\begin{remark}
The endpoint estimates can also be recovered, both in the massless and massive case, adapting the proof of Lemma 13 in \cite{IonescuKenig05-a}, but we will not go into details of this aspect.
\end{remark}

\section{Magnetic Hardy inequality}

This section is devoted to the proof of a version of Hardy's
inqeuality adapted to the perturbed Dirac operator
\begin{equation*}
  \mathcal{H}=\mathcal{D}_{A}+m \beta,\qquad
  \mathcal{D}_{A}=i^{-1}\alpha \cdot \nabla_{A}\equiv
       i^{-1}\alpha \cdot (\nabla-iA).
\end{equation*}
The proof is simple but we include it for the sake of completeness.

\begin{proposition}\label{Hardy}
Let $B=DA-DA^{t}=B_1+B_2$ and assume that
\begin{equation}\label{hardhp}
\||x|^2B_1\|_{L^\infty(\mathbb{R}^n)}<\infty,\quad \|B_2\|_{L^\infty(\mathbb{R}^n)}<\infty.
\end{equation}
Then for every $f:\mathbb{R}^n\rightarrow\mathbb{C}^M$ such that $\mathcal{H}f\in L^2$ and any $\varepsilon<1$ the following 
inequality holds when $m\neq0$:
\begin{equation}\label{hardyineq}
\begin{split}
  m^2\int_{\mathbb{R}^n}|f|^2
  +\left(
      (1-\varepsilon)
      \frac{(n-2)^2}{4}
      -\frac12\||x|^2B_1\|_{L^\infty}
    \right) &
  \int_{\mathbb{R}^n}
  \frac{|f|^2}{|x|^2}
  +\varepsilon\int_{\mathbb{R}^n}|\nabla_Af|^2\le
    \\
  \le 
  &
  \left(1+\frac{\|B_2\|_{L^\infty}}{2m^2}\right)
  \int_{\mathbb{R}^n}|\mathcal{H}f|^2.
\end{split}
\end{equation}
When $m=0$, the inequality is also true provided we choose
$B_{1}=B$, $B_{2}=0$ and we interpret the right hand side of
\eqref{hardyineq} simply as $\int|\mathcal{H}f|^{2}$.
\end{proposition}

\begin{proof}
Denote with $(\cdot,\cdot)$ the inner product in $L^2(\mathbb{R}^n,\mathbb{C}^M)$ and with $\|\cdot\|$ the 
associated norm. Recalling \eqref{Hsquare3}, we can write
\begin{equation*}
  \|\mathcal{H}f\|^{2}=
  m^{2}\|f\|^{2}+
  \|\nabla_{A}f\|^{2}+
  i(S \cdot B f,f)
\end{equation*}
where the matrix $S \cdot B=[S_{jk}]\cdot[B^{jk}]$
is skew symmetric since
\begin{equation*}
  S_{jk}=\frac12 \alpha_{j}\alpha_{k},\qquad
  B^{jk}=\partial_{j}A^{k}-\partial_{k}A^{j}.
\end{equation*}
The selfadjoint matrices $\alpha_{j}$ have norm less than 1
(recall $\alpha^{2}_{j}=\mathbb{I}$), so that
\begin{equation*}
  |(S \cdot Bf,f)|\le 
  \frac12([B]_{1}f,f)
\end{equation*}
where we denote by $[B]_{1}$ the $\ell^{1}$ matrix norm
\begin{equation*}
  [B(x)]_{1}=\sum_{j,k}|B^{jk}(x)|.
\end{equation*}
Now recalling assumption \eqref{hardhp} we can write
\begin{equation*}
  |(S \cdot Bf,f)|\le 
  \frac12 \||x|^{2}B_{1}\|_{L^{\infty}}\left\|\frac{f}{|x|}\right\|^{2}
  +
  \frac12\|B_{2}\|_{L^{\infty}}\|f\|^{2}
\end{equation*}
and in conclusion
\begin{equation*}
  \|\mathcal{H}f\|^{2}\ge 
  m^{2}\|f\|^{2}+
  \|\nabla_{A}f\|^{2}-
  \frac12 \||x|^{2}B_{1}\|_{L^{\infty}}\left\|\frac{f}{|x|}\right\|^{2}
  -
  \frac12\|B_{2}\|_{L^{\infty}}\|f\|^{2}.
\end{equation*}
We now recall the magnetic Hardy inequality proved in 
\cite{FanelliVega09-a}:

\begin{equation}\label{fanveahar}
\displaystyle\frac{(n-2)^2}{4}\int_{\mathbb{R}^n}\frac{|f|^2}{|x|^2}\leq \int_{\mathbb{R}^n}|\nabla_Af|^2.
\end{equation}
Observing now that
$$
\|\mathcal{H}f\|^2=(\mathcal{H}^2f,f)=m^2\|f\|^2+\|\mathcal{D}_Af\|^2
$$
and that
$$
\displaystyle (1-\varepsilon)\left\|\frac{f}{|x|}\right\|^2
+\varepsilon\|\nabla_Af\|^2\leq\|\nabla f\|^2,
$$
the proof is complete.
\end{proof}

\end{document}